\documentclass[reqno]{amsart}

\usepackage{amsmath}
\usepackage{amssymb}
\usepackage{amsfonts}
\usepackage{amsthm}

\usepackage{hyperref}

\usepackage{graphicx}
\usepackage{tikz, tikz-3dplot}

\usepackage{subcaption}

\usepackage{algorithm}
\usepackage{algpseudocode}

\newtheorem*{thm}{Theorem}
\newtheorem{lemma}{Lemma}

\newtheorem*{corollary}{Corollary}
\newtheorem*{conj}{Conjecture}

\newcommand{\SO}{\text{SO}}

\begin{document}

\title[]{A stellated tetrahedron that \\is probably not Rupert}
\author[]{Tony Zeng}
\address{Department of Mathematics, University of Washington, Seattle, WA 98195, USA}
\email{txz@uw.edu}

\begin{abstract} 
    A convex polyhedron is Rupert if a hole can be cut into it (making its genus \(1\)) such that an identical copy of the polyhedron can pass through the hole. Resolving a conjecture of Jerrard-Wetzel-Yuan, Steininger and Yurkevich recently constructed a convex polyhedron which is not Rupert. We propose a search for the simplest possible non-Rupert polyhedron and provide numerical evidence suggesting that a particular stellated tetrahedron is not Rupert.  The computational techniques utilize linear program solvers to compute the largest possible scalings of polygons that can be translated to fit in other polygons. The relative simplicity of the stellated tetrahedron as compared to other polyhedra allows this more rudimentary check to be computationally tractable. In particular, we show that over \(88\%\) of a particular encoding of \(\SO(3) \times \SO(3)\) equipped with the standard measure does not yield a Rupert passage.
\end{abstract}

\maketitle

\section{Introduction}

\subsection{Introduction}
According to John Wallis \cite{wallis}, Prince Rupert of the Rhine, whom Wallis calls \textit{vir magno ingenio \& sagacitate} [a man of great ingenuity and sagacity] posed the question (and wagered in favor of its answer being the affirmative) of whether a cube can ``pass through itself'', i.e. whether hole can be drilled into a cube such that an identical cube can be passed through that hole. 

\begin{figure}[h!]
    \centering
    \includegraphics[scale=0.25,trim={31cm 8.7cm 10cm 14cm},clip]{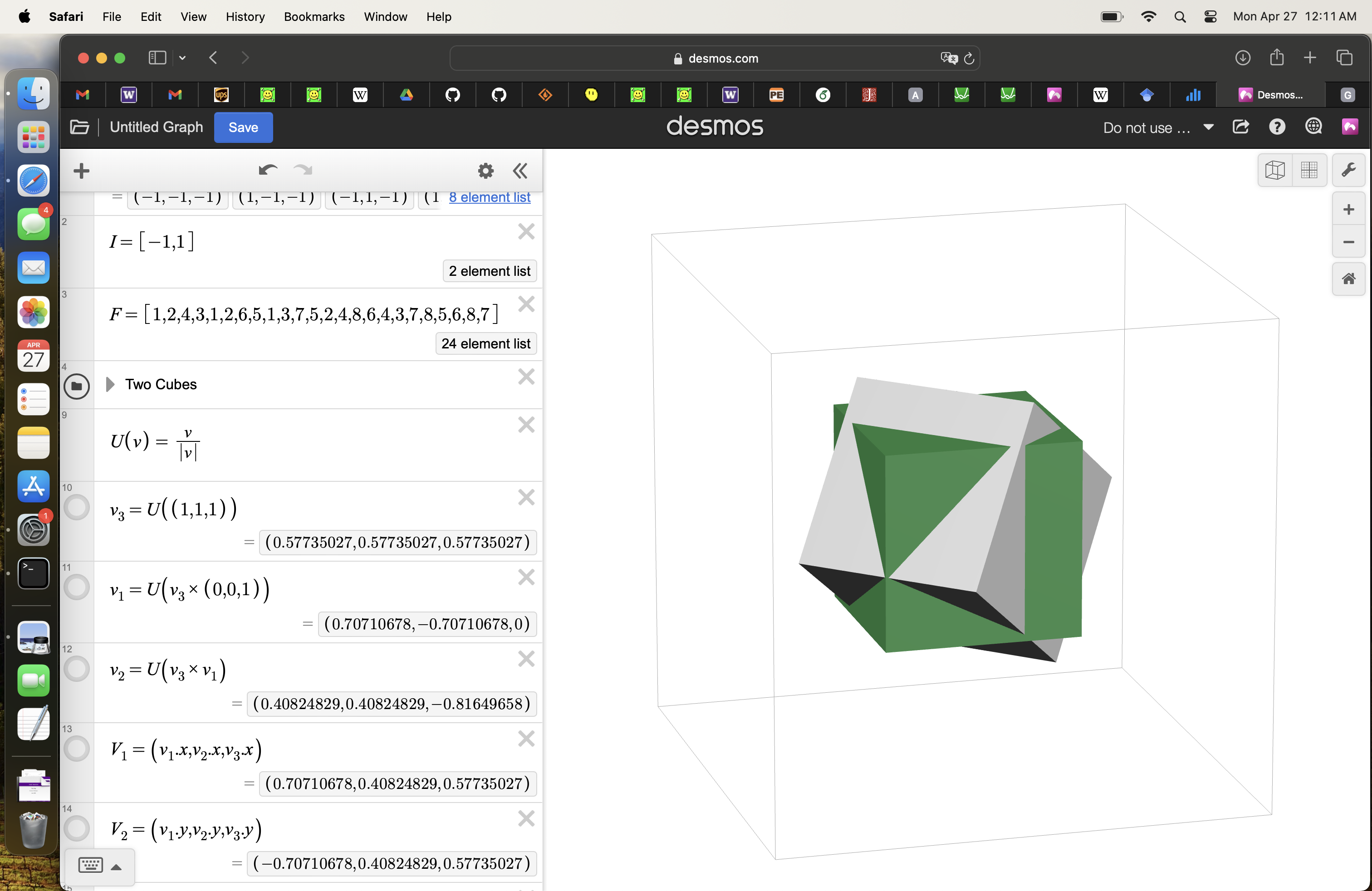} \(\qquad\)
    \includegraphics[scale=0.25,trim={31cm 8.7cm 10cm 14cm},clip]{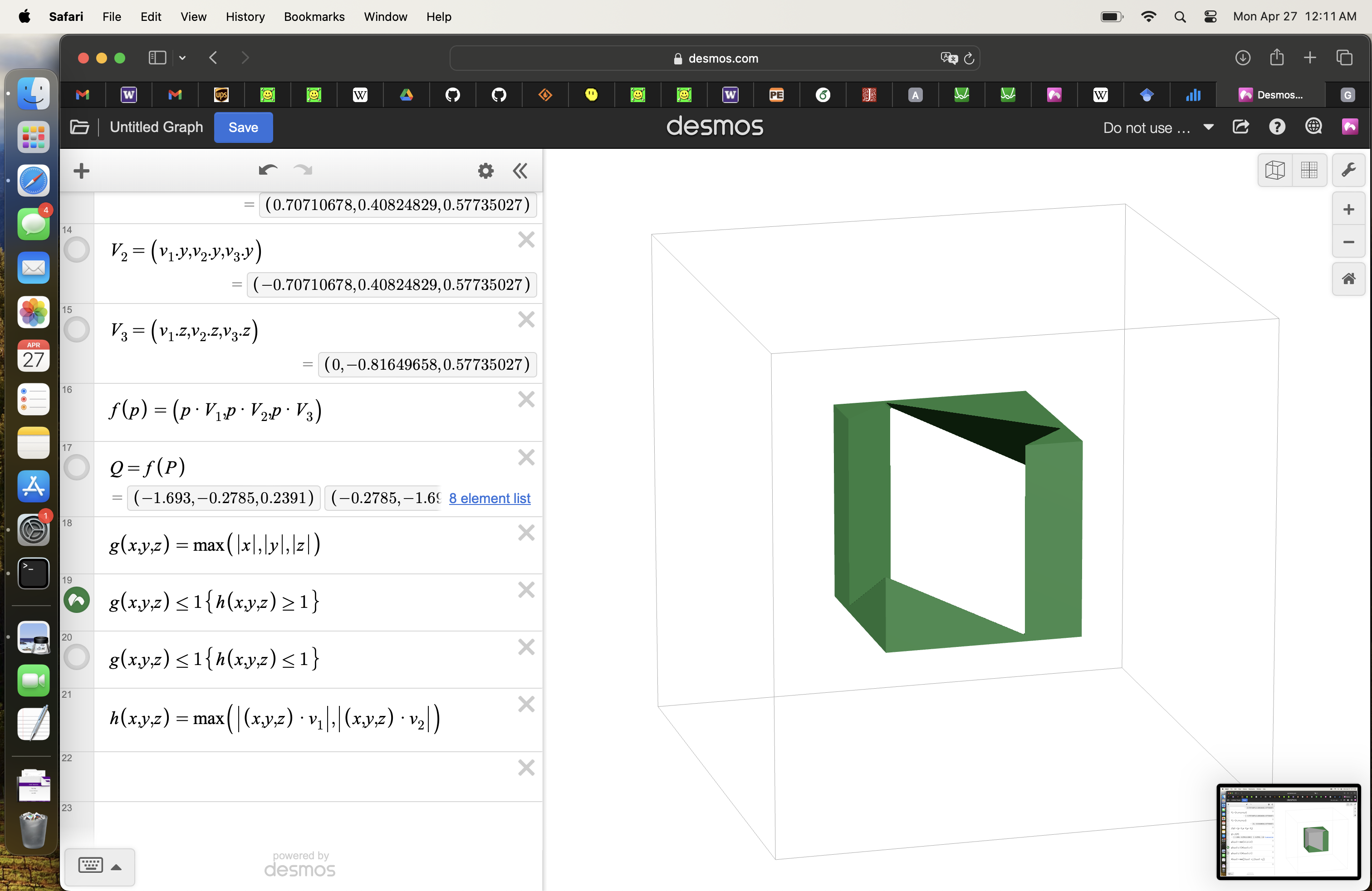}
    \caption{An illustration that the cube is Rupert. }
    \label{fig:cube}
\end{figure}

Rupert's question was answered in the affirmative by Wallis himself, who showed that the flat projection of the cube into a square can fit inside the regular hexagonal projection of the cube along its space diagonal (see Figure \ref{fig:cube}). In fact, Wallis's construction allows a cube with side length \(\sqrt{6} - \sqrt{2} \approx 1.03\) to pass through a hole drilled through a unit cube. Later, Peter Nieuwland (1764 -- 1794) showed that an by using a slightly different projection, the same is true of a cube with side length \(\sqrt{18} / 4 \approx 1.06\) and that this improvement is optimal. Nieuwland himself never published this; it was published posthumously by Nieuwland's advisor in \cite{van1816grondbeginsels}.

To make all of this precise, we introduce the following notation. We use \(\mathcal{P}\) to denote a polyhedron which refers to the convex hull of a finite set of points in \(\mathbb{R}^3\) in convex position. Moreover, let \(\mathcal{P}^\circ\) denote the interior of \(\mathcal{P}\). Let \(P : \mathbb{R}^3 \rightarrow \mathbb{R}^2\) denote the projection which drops the \(x\) coordinate. Then Prince Rupert asked, do there exist \(R_1, R_2 \in \SO(3)\) and \(t \in \mathbb{R}^2\) such that
\[ P \circ R_1(\mathcal{P}_C) + t \subset P \circ R_2(\mathcal{P}_C^\circ), \]
where \(\mathcal{P}_C\) is the set of vertices of the cube. Note that the use of of the interior on the right side of the containment is important, as otherwise, we could trivially choose \(R_1 = R_2\). We say that a polyhedron \(\mathcal{P}\) has the \textit{Rupert property} (or is \textit{Rupert}) if it satisfies the above condition, and that the orientations \(R_1, R_2\) are a \textit{Rupert passage} for \(\mathcal{P}\). Moreover, one can measure ``how Rupert'' a polyhedron \(\mathcal{P}\) is by testing what the largest scaling \(s\mathcal{P}\) is, that can still pass through \(\mathcal{P}\). This notion is known as the \textit{Nieuwland number} or \textit{Nieuwland constant} (see \cite{fredriksson, steininger2025convex}) and is defined more precisely as the largest \(s \in \mathbb{R}\) such that there exists \(R_1, R_2 \in \SO(3)\) and \(t \in \mathbb{R}^2\) for which
\[ P \circ R_1(s\mathcal{P}) + t \subset P \circ R_2(\mathcal{P}^\circ) \]
holds. For instance, as discussed above, Nieuwland showed that the Nieuwland number of the cube is \(\sqrt{18}/{4}\). 
\begin{figure}[h!]
    \centering
    \includegraphics[scale=0.4]{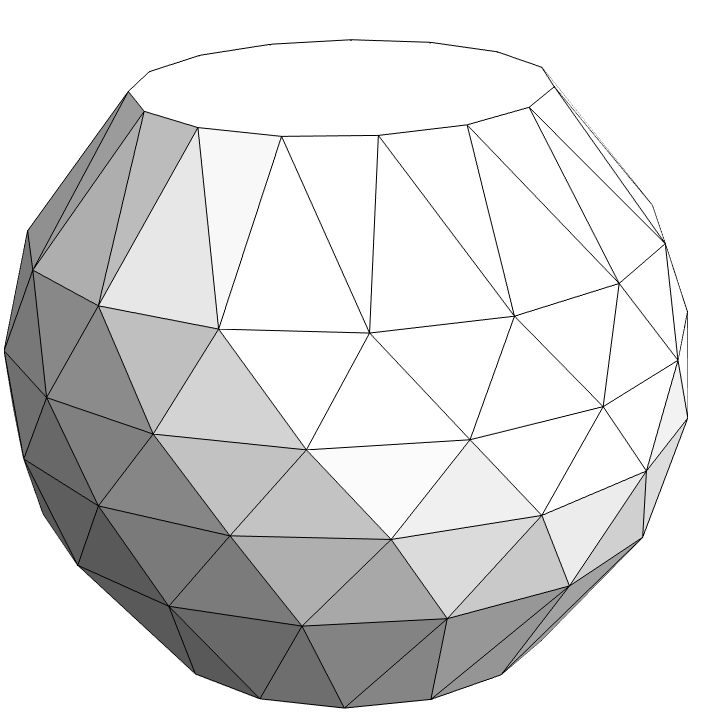}
    \caption{The Noperthedron of Steininger and Yurkevich.}
    \label{fig:nopert}
\end{figure}
\subsection{State of the Art}
Currently, it is known that all 5 Platonic solids, 11 out of 13 Catalan solids, and 87 of the 92 Johnson solids have the Rupert property. Note that proving that a polyhedron is Rupert is a far easier task than showing that it is not. For the former, one need only exhibit two projections and verify containment, while for the latter, it is not immediately clear how one might even begin. The fact that so many nice solids (as well as other polyhedra) have the Rupert property led to the following conjecture.
\begin{conj}[Jerrard-Wetzel-Yuan \cite{conjecture}]
    All convex polyhedra are Rupert.
\end{conj}
Jerrard, Wetzel, and Yuan \cite{conjecture} state their conjecture ``with a certain hesitancy''. This caution seems to have been appropriate, as recently, Steininger and Yurkevich \cite{steininger2025convex}, constructed a convex polyhedron that cannot pass through itself. 
\begin{thm}[Steininger-Yurkevich  \cite{steininger2025convex}]
    The Noperthedron is not Rupert.
\end{thm}
Their proof is extremely computational and relies on a cleverly constructed polyhedron (shown in Figure \ref{fig:nopert}) with \(90\) vertices and \(2\pi / 15\) rotational symmetry, allowing them to reduce the size of their parameter space. In particular, their polyhedron is centrally symmetric, meaning that for every vertex \(v\) of the polyhedron, \(-v\) is also a vertex. This condition effectively allows the removal of \(t\) in the above containment condition. Intuitively, the existence of the Noperthedron as a counterexample would suggest that there should exist many other potential non-Rupert polyhedra. Steininger and Yurkevich made the following conjecture \cite{steininger2023algorithmic} which, to the best of our knowledge, remains open.
\begin{conj}[Steininger-Yurkevich]
    The rhombicosidodecahedron is not Rupert.
\end{conj}
 The rhombicosidodecahedron is an Archimedean solid with 62 faces (20 triangles, 30 squares, and 12 pentagons), 60 vertices, and 120 edges. Our main motivation was that there should surely exist simpler counterexamples, i.e. polyhedra that are not Rupert and have many fewer vertices and a simpler face structure, leading us to ask the following.
\begin{quote}
    \textbf{Question.} What is the `simplest' polyhedron that is not Rupert? 
\end{quote}
\vspace{-10pt}
\subsection{The Stellated Tetrahedron}
The Triakis Tetrahedron is a polyhedron with 8 vertices, 18 edges, and 12 faces. It can be obtained by 
starting with a regular tetrahedron and then adding an additional point for each of the 4 faces; this point has the property that its orthogonal projection onto face is exactly the centroid of the triangle (this process is known as `stellation'). We stellate each face of a regular tetrahedron such that the resulting polyhedron has constant dihedral angle, i.e. the angle between any two adjacent faces is constant. Fredriksson showed in \cite{fredriksson} that the Triakis Tetrahedron is Rupert with a lower bound for its Nieuwland number to be roughly \(1 + 4 \times 10^{-6}\). The techniques used by Fredriksson likely ensure that this is not far off from the true Nieuwland number. 
 This value is remarkably low which led us to consider similar polyhedra as potential non-Rupert candidates.
\vspace{-5pt}
\begin{figure}[h!]
    \centering
    \includegraphics[scale=0.43]{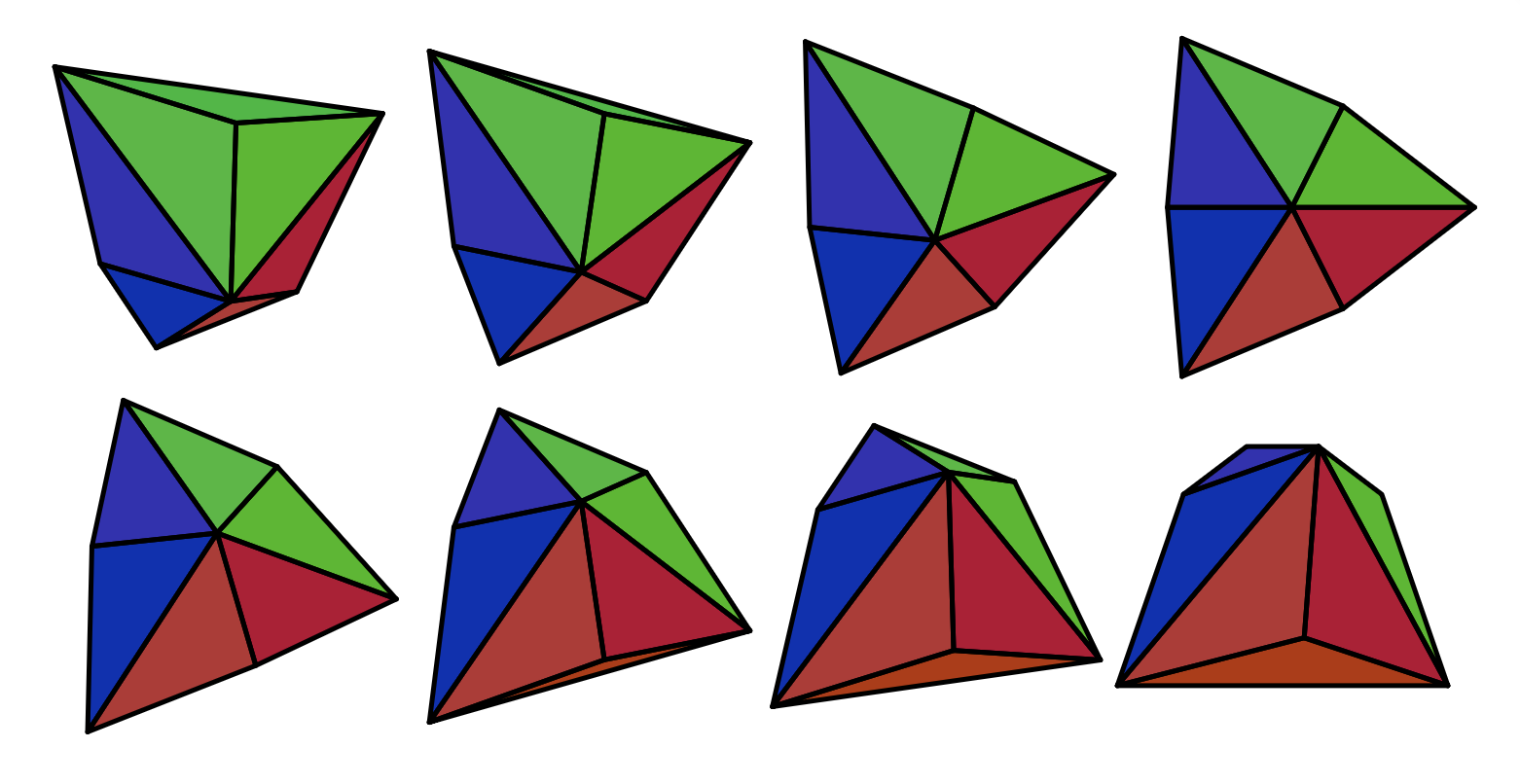}
    \caption{The stellated tetrahedron \(\mathcal{P}_{11/20}\). Faces with similar colors share the same base tetrahedral face prior to stellation. We conjecture this convex polyhedron to not be Rupert.} 
\end{figure}
\vspace{-5pt}

We consider the stellated tetrahedron \(\mathcal{P}_a\) with vertices
\begin{align*}
    p_0 &= (1, 1, 1), &q_0 &= (-a, -a, -a), \\
    p_1 &= (1, -1, -1), &q_1 &= (-a, a, a), \\
    p_2 &= (-1, 1, -1), &q_2 &= (a, -a, a), \\
    p_3 &= (-1, -1, 1), &q_3 &= (a, a, -a).
\end{align*}
Edges are between all \( p_i, p_j\) and \(p_i, q_j\) for \(i \ne j\), and \(a \in (0.5, 0.57)\). Note that \(a = 0.6\) corresponds to the Triakis tetrahedron. An algorithm implemented by Fredriksson \cite{fredriksson} is capable of quickly verifying the Rupert property for many polyhedra, including several that were previously not known to be Rupert. For stellated tetrahedra with \(a \in (0.5, 0.57)\), however, the code is unable to find a Rupert passage (see Figure \ref{fig:numerics}). Scott, in \cite{scott}, provides a sufficient condition for a polyhedron to be locally Rupert, which these stellated tetrahedra do not satisfy. We conjecture that these stellated tetrahedra are not Rupert. 
\begin{conj}
    The stellated tetrahedron \(\mathcal{P}_{11/20}\) is not Rupert.
\end{conj}
The choice of \(a = 11/20\) is, admittedly, somewhat arbitrary. It is a nice enough rational number in the interval \((0.5, 0.57)\) which is convenient to type as input into a computer program. In support of this conjecture, we provide numerical evidence that the majority of pairs of orientation do not yield a Rupert passage. 

\begin{figure}[h!]
    \centering
    \includegraphics[scale=0.3]{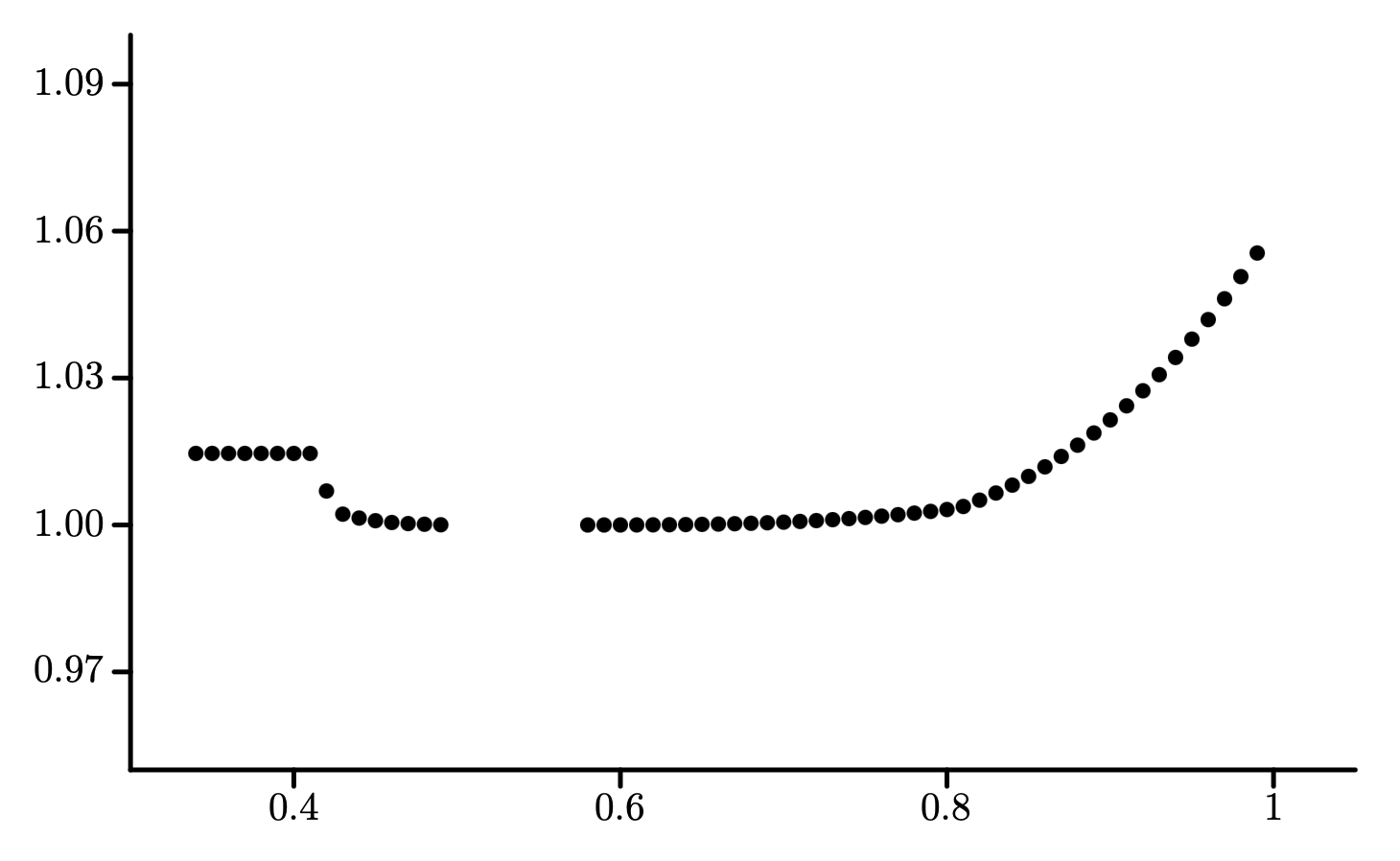}
    \caption{Numerical Nieuwland numbers for \(\mathcal{P}_a\) plotted against \(a\) computed using the techniques of \cite{fredriksson}.}
    \label{fig:numerics}
\end{figure}
\subsection{Main Result}

The space of all pairs of orientations, that is \(\SO(3) \times \SO(3)\), is a \(6\)-dimensional manifold. Since we can encode an element of \(\SO(3)\) using three angles \(\alpha, \theta, \phi\), we can define a parameter space \(\Lambda_0\) as a \(6\)-dimensional space, i.e. \([0, 2\pi]^6\) over which to search for Rupert passages. We discuss in \S 3 how \(\Lambda_0\) can be reduced to \(\Lambda = [0, 2\pi]^2 \times [0, \pi] \times [0, \pi / 2]^2\). If we use the product measure on \(\Lambda\), we can formulate the following statement.

\begin{thm}
    There is no Rupert passage for \(\mathcal{P}_{11/20}\) in \(88\%\) of \(\Lambda\).
\end{thm}

Additional computational efforts may suffice to do increase the number and establish the result for $(100 - \varepsilon) \%$ for any $\varepsilon > 0$. However, we lack a satisfactory local theory ensuring that `infinitesimal rotations' cannot create a valid Rupert passage, i.e. we cannot check that \(\mathcal{P}_{11/20}\) is not locally Rupert. It is not clear to us whether \(\mathcal{P}_{11/20}\), a polyhedron on 8 vertices, is the simplest example of a convex polyhedron that is not Rupert (if indeed it is not Rupert); however, it is a reasonable candidate.
 
\section{Lemmas}

Note that throughout this section, we refer to \(\mathcal{P}_a\) instead of \(\mathcal{P}_{11/20}\) as the statements and results are not specific to a particular value of \(a\).

\subsection{The Polygon Setting}

Given two orientations \(R_1, R_2\) of \(\mathcal{P}_a\), we would like to check whether \(R_1 \mathcal{P}_a\) can ``fit inside'' \(R_2 \mathcal{P}_a\). Moreover, if \(R_1 \mathcal{P}_a\) ``cannot fit inside'' \(R_2 \mathcal{P}_a\), we would hope to say that \(R'_1 \mathcal{P}_a\) also ``cannot fit inside'' \(R'_2 \mathcal{P}_a\) for orientations \(R'_1, R'_2\) that are ``close'' to \(R_1, R_2\) respectively. We now establish the necessary machinery to make all of this precise.
Let \(P, Q\) be convex polygons in \(\mathbb{R}^2\) with vertices \(\{p_i\}_{i \in [n]}\), \(\{q_j\}_{j \in [m]}\) (ordered counterclockwise) respectively. Like in the polyhedron setting, when we say polygon, we refer to the convex hull of a finite set of points in convex position in the plane. Let \(\{n_i\}\) denote the outward facing normals of \(Q\), i.e. 
\[ n_i = R_{\pi / 2}(q_i - q_{i + 1}). \]
Then \(x \in Q\) if and only if \(\langle n_i, x \rangle \le b_i\) for all \(i\), where \(b_i = \langle n_i, q_i\rangle\). Let \(s \in \mathbb{R}\) and \(t \in \mathbb{R}^2\), and consider the set of inequalities
\begin{equation}
\langle n_j, sp_i + t \rangle \le b_j, 
\end{equation}
for all \(i \in [n], j \in [m]\). Geometrically, values of \(s, t\) that satisfy this system of inequalities correspond to a scaling of \(P\) that can be translated to fit inside \(Q\), i.e. 
\begin{equation}
sP + t \subset Q^\circ. 
\label{eq:fits}
\end{equation}
We say that \textit{\(P\) fits inside \(Q\)} or \textit{\(Q\) contains \(P\)} if there exists \(t \in \mathbb{R}^2\) and \(s > 1\) satisfying \ref{eq:fits}. If \(s^*\) is the largest possible value of \(s\), we say that \textit{\(P\) fits inside \(Q\) up to scale \(s^*\)}. Note that this relation is transitive, i.e. for polygons \(A, B, C\) if \(A\) fits inside \(B\) and \(B\) fits inside \(C\), then \(A\) fits inside \(C\).
We remark that others use similar language to allow rotating \(P\) at the price of a more complicated verification \cite{chazelle}. In our setting, we drop this rotation in favor of a simpler condition due to the relative simplicity of our polyhedron.

\begin{figure}[h!]
    \centering
    \begin{subfigure}[c]{0.45\textwidth}
        \centering
        \begin{tikzpicture}
            \draw (1, 1) -- (-1, 1) -- (-1, -1) -- (1, -1) -- (1, 1);
            \draw (-4.4, 0) -- (-3, 1.4) -- (-1.6, 0) -- (-3, -1.4) -- (-4.4, 0);
            \node at (-3, 0) {\(P\)};
            \node[above] at (-1, 1) {\(Q\)};
            \draw[dashed] (-1, 0) -- (0, -1) -- (1, 0) -- (0, 1) -- (-1, 0);
            \node at (0, 0) {\(\frac{5}{7}P\)};
        \end{tikzpicture}
    \end{subfigure}
    \begin{subfigure}[c]{0.45\textwidth}
        \centering
        \begin{tikzpicture}
            \draw (1, 1) -- (-1, 1) -- (-1, -1) -- (1, -1) -- (1, 1);
            \draw[white] (-4.4, 0) -- (-3, 1.4) -- (-1.6, 0) -- (-3, -1.4) -- (-4.4, 0);
            \draw (-3.2, 0) -- (-2.5, 0.7) -- (-1.8, 0) -- (-2.5, -0.7) -- (-3.2, 0);
            \node at (-2.5, 0) {\(P'\)};
            \node[above] at (-1, 1) {\(Q\)};
            \draw[dashed] (-1, 0) -- (0, -1) -- (1, 0) -- (0, 1) -- (-1, 0);
            \node at (0, 0) {\(\frac{10}{7}P\)};
        \end{tikzpicture}
    \end{subfigure}
    \caption{\(P\) does not fit inside \(Q\), but \(P'\) does fit inside \(Q\).}
\end{figure}
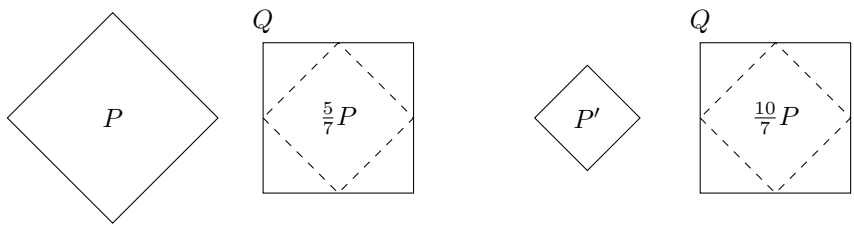

For a convex polygons \(P\), we say that a \(\delta\)-perturbation of \(P\) is any polygon \(P'\) such that \(\Vert p_i - p'_i\Vert < \delta\) for all \(i\). Then we have the following lemma.

\begin{lemma}
    For convex polygons \(P, Q\), let \(P'\) be a polygon that is contained by all \(\delta_1\)-perturbations of \(P\) and \(Q'\) be a polygon that contains all \(\delta_2\)-perturbations of \(Q\). Let \(s^*\) be the optimal solution to the linear program with objective function \(s\) and constraints
    \[ \langle n_j, sp'_i + t \rangle \le \langle n_j, q'_j\rangle, \]
    for all \(i \in [n], j \in [m]\). If \(s^* < 1\), then no \(\delta_1\)-perturbation of \(P\) can fit inside any \(\delta_2\)-perturbation of \(Q\).
\end{lemma}

\begin{proof}
    If \(s^* < 1\), then \(P'\) does not fit inside \(Q'\). Let \(\widetilde{P}\) be a \(\delta_1\)-perturbation of \(P\) and \(\widetilde{Q}\) be a \(\delta_2\)-perturbation of \(Q\). Then since \(P'\) fits inside \(\widetilde{P}\) and \(\widetilde{Q}\) fits inside \(Q'\), we immediately have that \(\widetilde{P}\) cannot fit inside \(\widetilde{Q}\).
\end{proof}

\subsection{The Polyhedron Setting}

We encode our orientations in the following way. Let \(M_{\theta, \phi}\) be the rotation matrix given by
\begin{align*}
M_{\theta, \phi} &= \begin{bmatrix}
    -\sin\theta & \cos\theta & 0 \\
    -\cos\theta \cos\phi & -\sin\theta \cos\phi & \sin\phi
\end{bmatrix} \\
& = \begin{bmatrix}
    0 & 1 & 0 \\
    0 & 0 & 1
\end{bmatrix} \begin{bmatrix}
    \sin\phi & 0 & \cos\phi \\
    0 & 1 & 0 \\
    -\cos\phi & 0 & \sin\phi
\end{bmatrix} \begin{bmatrix}
    \cos\theta & \sin\theta & 0 \\
    -\sin\theta & \cos\theta & 0 \\
    0 & 0 & 1
\end{bmatrix}, 
\end{align*}
and let \(R_\alpha\) denote counterclockwise rotation by \(\alpha\) in the plane. Then a polyhedron \(\mathcal{P}\) is Rupert if and only if there exist angles \(\alpha, \theta, \phi, \alpha', \theta', \phi'\) such that \(P = R_{\alpha} \circ M_{\theta, \phi} \mathcal{P}_a\) fits inside \(Q = R_{\alpha'} \circ M_{\theta', \phi'}\mathcal{P}_a\). Notice that we can always assume \(\alpha' = 0\), since we can rotate \(P\) by \(-\alpha'\). This allows us to encode \(\SO(3) \times \SO(3)\) as a 5-dimensional rectangular prism, i.e. \([0, 2\pi]^5\). If we have that \(P = R_{\alpha} \circ M_{\theta, \phi} \mathcal{P}_a\) fits inside \(Q = M_{\theta', \phi'} \mathcal{P}_a\), we say that \((\alpha, \theta, \phi, \theta', \phi')\) is a \textit{Rupert passage}. 

\begin{lemma}[Steininger-Yurkevich]
    For \(\varepsilon > 0\), if \(|\alpha - \alpha'|, |\theta - \theta'|, |\phi - \phi'| \le \varepsilon\), then
    \[ \Vert M_{\theta, \phi} - M_{\theta', \phi'} \Vert \le \sqrt{2} \varepsilon, \quad \text{and} \quad \Vert R_\alpha M_{\theta, \phi} - R_{\alpha'}M_{\theta', \phi'} \Vert \le \sqrt{5}\varepsilon. \]
\end{lemma}

For \(P = R_\alpha \circ M_{\theta, \phi} \mathcal{P}_a\), we say that \(P' = R_{\alpha'} \circ M_{\theta', \phi'}\) is an \(\varepsilon\)-angle perturbation of \(P\) if \(|\alpha - \alpha'|, |\theta - \theta'|, |\phi - \phi'| \le \varepsilon\). Then with these two lemmas, we obtain as a corollary the following.

\begin{figure}[h!]
    \centering
    \includegraphics[scale=0.4]{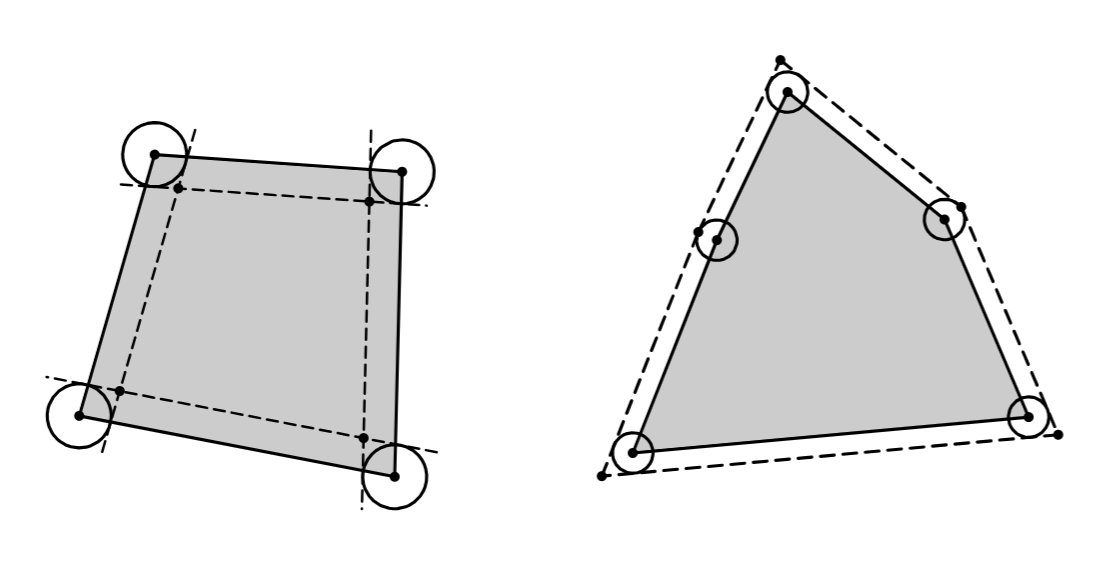}
    \caption{The construction of polygons \(P'\) (left) and \(Q'\) (right) as in the Corollary.}
    \label{fig:buffer}
\end{figure}

\begin{corollary}
    Let \(P = R_\alpha \circ M_{\theta, \phi} \mathcal{P}_a\) and \(Q= M_{\theta', \phi'} \mathcal{P}_a\). Let \(P'\) be a polygon that is contained by all \(\varepsilon\)-angle perturbations of \(P\) and \(Q'\) be a polygon that is contained by all \(\varepsilon\)-angle perturbations of \(Q\) and suppose \(P'\) fits inside \(Q'\) up to scale \(s\). If \(s < 1\), then there is no Rupert passage for any \(\varepsilon\)-angle perturbations of \(P, Q\).
    \label{cor:check}
\end{corollary}

So long as we have a way of constructing such \(P'\) and \(Q'\), this allows us a way of verifying that there is no Rupert passage within a full \(\varepsilon\) side-length cube in our 5-dimensional parameter space. One can do so using tangent lines to circles whose radii are the perturbation parameter, as in Figure \ref{fig:buffer}. 

\subsection{Symmetries}

In practice, \([0, 2\pi]^5\) is rather large, since the cubes we end up computing have small side length, in particular order \(10^{-1}\) or smaller. Because the stellated tetrahedron exhibits tetrahedral symmetry, there are a number of symmetries that can reduce our parameter space and computation time. 

\begin{lemma}
    \(M_{\theta + \pi / 2,\phi} \mathcal{P}_a = -M_{\theta, \phi}\mathcal{P}_a = R_\pi \circ M_{\theta, \phi}\mathcal{P}_a\).
    \label{lem:theta}
\end{lemma}

\begin{proof}
    Compute
    \[ M_{\theta, \phi}(b) = \begin{bmatrix}
        -\sin\theta & \cos\theta & 0 \\
        -\cos\theta \cos\phi & -\sin\theta \cos\phi & \sin\phi
    \end{bmatrix} \begin{pmatrix}
        b_x \\ b_y \\ b_z
    \end{pmatrix} \]
    \[ = \begin{pmatrix}
        -b_x\sin\theta + b_y\cos\theta \\
        \cos\phi(-b_x\cos\theta - b_y\sin\theta) + b_z\sin\phi
    \end{pmatrix}, \]
    and
    \[ M_{\theta + \pi / 2, \phi}(c) = \begin{bmatrix}
        -\cos\theta & -\sin\theta & 0 \\
        \sin\theta\cos\phi & -\cos\theta\cos\phi & \sin\phi
    \end{bmatrix} \begin{pmatrix} 
        c_x \\ c_y  \\  c_z
    \end{pmatrix} \]
    \[ = \begin{pmatrix}
        -c_x\cos\theta -c_y\sin\theta \\
        \cos\phi(c_x\sin\theta - c_y\cos\theta) + c_z\sin\phi
    \end{pmatrix}. \]
    Then if \(c_x = -b_y\), \(c_y = b_x\), and \(c_z = b_z\), we have that \(M_{\theta, \phi}(b) = M_{\theta + \pi/2, \phi}(c)\). In particular, one can check that
    \begin{align*}
        M_{\theta,\phi}(p_0) &= -M_{\theta + \pi / 2, \phi}(p_1), &M_{\theta,\phi}(q_0) &= -M_{\theta + \pi / 2, \phi}(q_1), \\
        M_{\theta,\phi}(p_1) &= -M_{\theta + \pi / 2, \phi}(p_3), &M_{\theta,\phi}(q_1) &= -M_{\theta + \pi / 2, \phi}(q_3), \\
        M_{\theta,\phi}(p_2) &= -M_{\theta + \pi / 2, \phi}(p_0), &M_{\theta,\phi}(q_2) &= -M_{\theta + \pi / 2, \phi}(q_0), \\
        M_{\theta,\phi}(p_3) &= -M_{\theta + \pi / 2, \phi}(p_2), &M_{\theta,\phi}(q_3) &= -M_{\theta + \pi / 2, \phi}(q_2).
    \end{align*}
\end{proof}

Let \(\widehat{\theta} = \theta \pmod{\pi / 2}\) and \(\widehat{\theta'} = \theta' \pmod{\pi / 2}\). Checking whether \(P = R_\alpha \circ M_{\theta, \phi} \mathcal{P}_a\) fits in \(Q = M_{\theta', \phi'}\mathcal{P}_a\) or vice versa is equivalent to checking whether \(P_2\) fits in \(Q_2\) or vice versa, where \(P_2 = R_{\alpha + \beta} \circ M_{\widehat{\theta}, \phi} \mathcal{P}_a\) and \(Q_2 = M_{\widehat{\theta'}, \phi'}\) for some angle \(\beta \in \{0, \pi\}\). Consequently, we may restrict both \(\theta\) parameters to the interval \([0, \pi / 2]\), thereby reducing our parameter space to \([0, 2\pi]^3 \times [0, \pi / 2]^2\).

\begin{lemma}
    \(M_{\theta, \phi + \pi}\mathcal{P}_a = U \circ M_{\theta, \phi}\mathcal{P}_a\), where \(U\) denotes reflection across the \(x\)-axis in the plane.
\end{lemma}

\begin{proof}
    Compute
    \[ M_{\theta, \phi + \pi} = \begin{bmatrix}
        -\sin\theta & -\cos\theta & 0 \\
        \cos\theta\cos\phi & \sin\theta\cos\phi & -\sin\phi
    \end{bmatrix} \]
    \[ = \begin{bmatrix}
        1 & 0 \\ 0 & -1
    \end{bmatrix} \circ \begin{bmatrix}
        -\sin\theta & \cos\theta & 0 \\
        -\cos\theta \cos\phi & -\sin\theta \cos\phi & \sin\phi
    \end{bmatrix} = U \circ M_{\theta, \phi}. \]
\end{proof}

Let \(\widehat{\phi} = \phi \pmod \pi\) and \(\widehat{\phi'} = \phi' \pmod \pi\). If both \(\phi, \phi' > \pi\), then checking whether \(P = R_\alpha \circ M_{\theta, \phi} \mathcal{P}_a\) fits in \(Q = M_{\theta', \phi'}\) or vice versa is equivalent to checking the same for \(P_2, Q_2\), where \(P_2 = R_\alpha \circ M_{\theta, \widehat{\phi}} \mathcal{P}_a\) and \(Q_2 = M_{\theta', \widehat{\phi'}}\). Therefore, we need not have both \(\phi, \phi'\) range over the entirety of \([0, 2\pi]\). In particular, we can restrict \(\phi\) to \([0, \pi]\), thereby reducing our parameter space again by an additional factor of \(2\), yielding a parameter space of \(\Lambda = [0, 2\pi]^2 \times [0, \pi] \times [0, \pi / 2]^2\).

\section{Computational Approach}

\subsection{Implementation Details}

There are a number of technical details that require discussion. First, computing \(s^*\) such that \(P\) fits in \(Q\) up to scale \(s^*\) is achieved by an LP solver. We use \verb|fitScale| to denote a method that takes two polygons \(P\) and \(Q\) and returns \(s^*\). This is a fundamental computation that is used many times when implementing the Corollary. To do so, we implement the subroutine \verb|Check| (see Algorithm \ref{alg:check}), which takes as input a 5-tuple \((\alpha, \theta, \phi, \theta', \phi')\) and a threshold \(b\) and returns a value \(\varepsilon\) such that no Rupert passage is possible within the \(\varepsilon\) side-length box centered on \((\alpha, \theta, \phi, \theta', \phi')\). 
\begin{algorithm}
    \caption{Check}
    \begin{algorithmic}
        \State \(P \gets R_{\alpha} \circ M_{\theta, \phi} \mathcal{P}_{11/20}\)
        \State \(Q \gets M_{\theta', \phi'} \mathcal{P}_{11/20}\)
        \State \(s_P \gets \verb|fitScale|(P,Q)\)
        \State \(s_Q \gets \verb|fitScale|(Q,P)\)
        \If{\(\max(s_P,s_Q) > 1\)}
        
            \Return INVALID
        \EndIf
        \State \(\delta \gets 0.25\)
        \While{\(\delta > b\)}
            \State \(P_- \gets \verb|buffer|(P, -\sqrt{5}\delta)\)
            \State \(Q_+ \gets \verb|buffer|(Q, \sqrt{2}\delta)\)
            \State \(P_+ \gets \verb|buffer|(P, \sqrt{5}\delta)\)
            \State \(Q_- \gets \verb|buffer|(Q, -\sqrt{2}\delta)\)
            \State \(\delta \gets 0.8\delta\)
            \State \(s_- \gets \verb|fitScale|(P_-, Q_+)\)
            \State \(s_+ \gets \verb|fitScale|(P_+, Q_-)\)
            \If{\(\max(s_-, s_+) < 1\)}
                \State \Return \(2\delta\)
            \EndIf
        \EndWhile
        \State \Return \(0\)
    \end{algorithmic}
    \label{alg:check}
\end{algorithm}

This is achieved by starting with a large initial perturbation value \(\delta\), which will result in \(P'\) and \(Q'\) polygons that are too disparate in size allowing on to fit into the other. We then iteratively reduce \(\delta\) until \(P'\) and \(Q'\) cannot fit into each other or until \(\delta\) is below our threshold \(b\). Note that the pairs \((P_-, Q_+)\) and \((P_+, Q_-)\) play the roles \(P'\) and \(Q'\) in the Corollary. We need two checks here to ensure that there is no containment in either direction. The \verb|buffer| method computes buffer polygons as illustrated in Figure \ref{fig:buffer} using the \verb|shapely| Python package.
With \(\verb|Check|\) described, we are now equipped to implement our main algorithm, whose aim is illustrated in Figure \ref{fig:cube_decomp} and pseudocode in Algorithm \ref{alg:orth}. 
\begin{algorithm}
    \caption{Orthtree Search}
    \begin{algorithmic}
        \State \(L \gets \{C_1, \ldots, C_{32}\}\)
        \State \(S \gets \{\}\)
        \While{\(|L| > 0\)}
            \State Pop cube \(C\) from \(L\)
            \State \(w \gets\) side length of \(C\)
            \If{\(C\) has side length less than \(\varepsilon\)}
                \State \(S \gets \verb|Join|(S, \{C\})\)
            \Else
                \State \(\varepsilon \gets \verb|Check|(C)\)
                \If{\(\varepsilon < w\)}
                    \State \(\{C_1', \ldots, C_{32}'\} \gets \verb|Split|(C)\)
                    \State \(L \gets \verb|Join|(L, \{C_1', \ldots, C_{32}'\})\)
                \EndIf
            \EndIf
        \EndWhile
        
        \State \Return{\(S\)}
    \end{algorithmic}
    \label{alg:orth}
\end{algorithm}

Our parameter space \(\Lambda\) can be decomposed into 32 cubes. Initialize a queue \(L\) with these 32 cubes \(\{C_1, \ldots, C_{32}\}\). We maintain a list \(S\) of boxes with side length less than \(b\). For each cube we apply \(\verb|Check|\) to see if we can rule out any Rupert passages in the entire cube. If we can, we move on to the next cube. If we cannot, subdivide the cube into smaller cubes and append them to our queue. What this effectively does is partition our space in the manner of an orthtree, the arbitrary dimension generalization of quadtrees and octrees. This will rule out the regions of our parameter space that are sufficiently far from the \(3\)-dimensional diagonal submanifold of \(\SO(3) \times \SO(3)\).
Here, \verb|Join| appends queues or sets together, and \verb|Split| takes as input a cube \(C\) and returns the \(32\) cubes obtained by splitting \(C\) in half along each coordinate direction. Running our algorithm yields our main result. 


\begin{figure}[h!]
    \centering
    \includegraphics[scale=0.5]{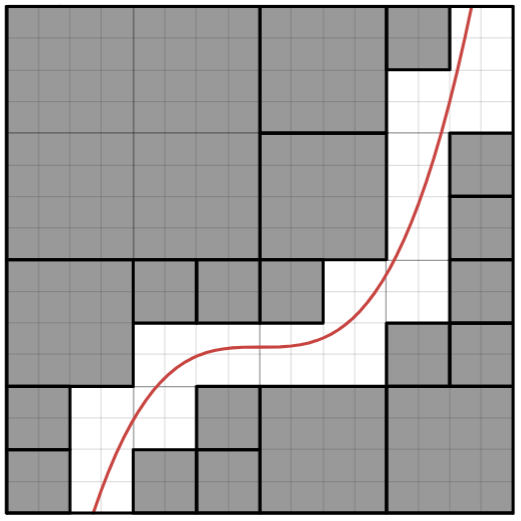}
    \caption{A \(2\)-dimensional illustration of the algorithm.}
    \label{fig:cube_decomp}
\end{figure}

\subsection{Discussion}
With some mild parallelization and a value of \(w = 0.08\) (yielding cubes of side length \(\pi / 64 \approx 0.049\) in \(S\)), our algorithm ran in roughly two days. Reducing \(w\) would greatly increase increase this runtime, due to the exponential factor of \(32\) whenever \verb|Split| is called. This would also increase the amount of \(\Lambda\) ruled to have no Rupert passages, though it is not clear by how much, as we have verified that some \verb|Check| does struggle with some pairs of orientations even if they yield very different projections. This is likely due to the fact that the polygon constructions for the Corollary are too aggressive. 
However, even with unlimited computational power, our techniques would be insufficient to show that \(\mathcal{P}_{11/20}\) is not Rupert as we lack a way to show that \(\mathcal{P}_{11/20}\) is not locally Rupert. That is to say, we do not have a result of the form \textit{for all \(R_1 \in \SO(3)\), there is no Rupert passage \(R_1, R_2\) for any orientation \(R_2\) that is \(\varepsilon\)-close to \(R_1\)}. Such a statement exists in \cite{steininger2025convex} as their polyhedron has a particular nice structure, and perhaps most importantly, is centrally symmetric. This allows them to removing the translation term in the containment condition \ref{eq:fits}, which greatly simplifies local analysis. We do not have a local theorem that accounts for the translation term. Nevertheless, we conjecture that this stellated tetrahedron is not Rupert.

\bibliographystyle{plain}
\bibliography{refs}

\end{document}